\documentclass[12pt]{article}
\usepackage{mathrsfs}
\usepackage{amsmath,amssymb,amsfonts,latexsym}
\newtheorem{Theorem}{Theorem}[section]
\newtheorem{Problem}{Problem}[section]

\newtheorem{Lemma}{Lemma}[section]

\newenvironment {proof} {\noindent{\em Proof.}}{\hspace*{\fill}$\Box$\par\vspace{4mm}}

\allowdisplaybreaks[4]

\begin{document}

\title{On spectral radius of strongly connected  digraphs}

\author{Jianping Li$^{a, b}$, Bo Zhou$^{ a}$\footnote{Corresponding author. e-mail: {\tt zhoubo@scnu.edu.cn}}
\\
$^a$Department of Mathematics,\\
South China Normal University,\\
Guangzhou 510631, P. R. China\\
$^b$Faculty of Applied
Mathematics,\\
Guangdong University of Technology,\\
Guangzhou 510090, P. R. China 
}
\date{}
\maketitle

\begin{abstract}
We determine  the digraphs which achieve the second, the third
and the fourth  minimum spectral radii respectively among  strongly connected
digraphs of order $n\ge 4$, and thus we answer affirmatively the problem whether 
the unique digraph which
achieves the minimum spectral radius among all strongly connected
bicyclic digraphs of order $n$ achieves the second minimum spectral
radius among all strongly connected digraphs of order $n$ for $n\ge 4$ proposed in 
  [H. Lin, J. Shu, A note on the spectral characterization
of strongly connected bicyclic digraphs, Linear Algebra Appl. 436
(2012) 2524--2530]. We also discuss the strongly connected bicyclic
digraphs  with small and large spectral radii respectively.
\end{abstract}

\section{ Introduction}

We consider digraphs without loops and multiple arcs.  Let $D$ be
 a digraph of order $n$ with vertex set $V(D)$ and arc set
$E(D)$. Let $V(D)=\{v_1, v_2,\dots, v_n\}$. The adjacency matrix of
$D$ is the $(0, 1)$-matrix $A(D)=(a_{ij})$ of order $n$ where
$a_{ij}=1$ if there is an arc from $v_i$ to $v_j$, and $a_{ij}=0$
otherwise. The eigenvalues of $D$ are the eigenvalues of $A(D)$. The
spectral radius of $D$ is  the largest modulus of an eigenvalue of
$D$, denoted by $\rho(D)$. Obviously, the eigenvalues of $D$ are the roots of the characteristic polynomial
of  $D$, denoted by $P(D, x)$, defined to be the characteristic
polynomial of the matrix $A(D)$, which is $\det(x I_n-A(D))$, where $I_n$ is the identity matrix of order $n$.
The spectra of digraphs have been studied to some extent, see e.g.,~\cite{R,EF,HJ,LS2}.

%

A digraph $D$ is strongly connected if for every pair $x,y \in
V(D)$, there exists a directed path from $x$ to $y$ and a directed
path from $y$ to $x$.  $D$ is called a strongly connected bicyclic
digraph if $D$ is strongly connected with $|E(D)|=|V(D)|+1$. For $n\ge 3$, let $\mathbb{B}_n$ be the set of strongly connected
bicyclic digraphs of order $n$.

Note that $D$
is strongly connected if and only if $A(D)$ is irreducible. It follows
from the Perron-Frobenius Theorem  that if $D$ is strongly connected, then $\rho(D)$ is an eigenvalue of $D$ and  there is a corresponding eigenvector whose coordinates are all positive.

Let $P_n$ be a directed path of order $n$. If $P_n=u_1u_2\dots u_n$,
then $u_1$ is the initial vertex, and $u_n$ is the terminal vertex
of $P_n$. The  $\theta$-digraph with parameters $a$, $b$ and $c$ with $a\le
b$,  denoted by $\theta(a, b, c)$,  consists of three directed paths
$P_{a+2}$, $P_{b+2}$ and $P_{c+2}$ such that the initial vertex of
$P_{a+2}$ and $P_{b+2}$ is the terminal vertex of $P_{c+2}$, and the
initial vertex of $P_{c+2}$ is the terminal vertex of $P_{a+2}$ and
$P_{b+2}$. These three directed paths are called the basic directed paths of $\theta(a, b, c)$. A $\infty$-digraph  with parameters $k$ and $l$ with
$k\le l$, denoted by $\infty(k, l)$, consists of two directed cycles
with exactly a vertex in common.

Recently, Lin and Shu \cite{HJ} showed that $\theta(0, 1, n -3)$ ($\infty(2,
n-1)$, respectively) is  the unique digraph in $\mathbb{B}_n$
which achieve the  minimum (maximum, respectively) spectral radius for $n\ge 4$.
Let $C_n$ be the directed cycle of order $n$. Note that $C_n$ uniquely achieves the minimum spectral
radius among all strongly connected digraphs of order $n\ge 3$.  Lin and Shu \cite{HJ} proposed the
following problem.
\begin{Problem} \label{p1}
Is  $\theta(0, 1, n-3)$ achieving the second minimum spectral
radius among all $n$-vertex strongly connected digraphs for $n\ge 4$?
\end{Problem}

Now we determine  the unique  digraphs which achieve the
second, the third and the fourth  minimum spectral radii
respectively  among strongly connected  digraphs of order $n\ge 4$, and thus
we answer Problem \ref{p1} affirmatively. To obtain this, we also determine the unique  digraphs in $\mathbb{B}_n$
with the
second and the third minimum spectral radii respectively  for $n\ge 4$. Finally, we determine the unique digraph in $\mathbb{B}_n$ with the second maximum spectral
radius for $n\ge 4$.

\section{Preliminaries}

We list some lemmas that will be used in our proof.

\begin{Lemma}\label{l2}\cite{HJ}
For $n\ge 4$, 
$\theta(0, 1, n-3)$ is the unique digraph in $\mathbb{B}_n$ which achieves the minimum
spectral radius, and  $\infty\left(\left\lfloor
\frac{n+1}{2}\right\rfloor,
\left\lceil\frac{n+1}{2}\right\rceil\right)$ is the unique $\infty$-digraph in $\mathbb{B}_n$
which achieves the minimum spectral radius.
\end{Lemma}

The following lemma was proved in \cite{HJ} for $c\ge 1$. By proof
there, it is also true for $c=0$.

\begin{Lemma} \label{add}\cite{HJ}
%
If $b\ge 1$, then 
$\rho(\theta(a,b,c))> \rho(\theta(a,b-1,c+1))$.
If $a\ge 1$, then 
$\rho(\theta(a,b,c))>\rho(\theta(a-1,b,c+1))$.
\end{Lemma}


The Following lemma was given in plain text in \cite{HJ} and it is a consequence of the well known coefficients theorem for digraphs, see e.g.,  \cite[Theorem 1.2, p.~36]{CDS}.

\begin{Lemma}  \label{poly} 
$P(\theta(a,b,c),x)=x^n-x^a-x^b$ with $n=a+b+c+2$, and $P(\infty(k,l),
x)=x^n-x^{k-1}-x^{l-1}$ with $n=k+l-1$.
\end{Lemma}


\begin{Lemma} \label{addd} For $n\ge 4$,
$\rho(\infty(\lfloor \frac{n+1}{2}\rfloor,
\lceil\frac{n+1}{2}\rceil))>\rho(\theta(0, 2, n-4))$.
\end{Lemma}

\begin{proof} Let $D_1=\theta(0, 2, n-4)$ and $D_2=\infty(\lfloor \frac{n+1}{2}\rfloor,
\lceil\frac{n+1}{2}\rceil)$. By Lemma \ref{poly},
$P(D_1,x)=x^n-x^2-1$ and  $P(D_2,x)=x^n-x^{\lfloor
\frac{n-1}{2}\rfloor}-x^{\lceil\frac{n-1}{2}\rceil}$. For $x\ge
\rho(D_2)>1$, $P(D_1,x)-P(D_2,x)=-x^2-1+2x^{\frac{n-1}{2}}\geq
-x^2-1+2x^2=x^2-1>0$ if $n$ is odd  and
$P(D_1,x)-P(D_2,x)=-x^2-1+x^{\frac{n}{2}-1}+x^{\frac{n}{2}}\geq
-x^2-1+x+x^2=x-1>0$  if $n$ is even.  Thus $\rho(D_2)>\rho(D_1)$.
\end{proof}

\begin{Lemma}\label{l5}  $\rho(\theta(0, 2, n-4))$
is strictly decreasing  for $n\ge 4$.
\end{Lemma}
\begin{proof}
Suppose that $n_1>n_2\ge 4$.  By Lemma \ref{poly},   $P(\theta(0,2,
n_1-4),x)-P(\theta(0, 2, n_2-4),x)= x^{n_1}-x^{n_2}>0$ for $x\ge
\rho(\theta(0, 2, n_2-4))>1$. Thus $\rho(\theta(0, 2,
n_1-4))<\rho(\theta(0, 2, n_2-4))$.
\end{proof}


Recall that the spectral radius of a nonnegative irreducible matrix $B$ is larger than that of a principal submatrix of $B$ and it increases when an entry of $B$ increases \cite[p.~19, 38]{H}. Thus we have the following well known lemma.

\begin{Lemma} \label{ad}
Let $D$ be a strongly connected digraph and $H$ a strongly connected
proper subdigraph of $D$. Then
$\rho(D)>\rho(H)$.
\end{Lemma}

\begin{Lemma}\label{ll2}  $\rho(\theta(0, n-2, 0))$
is strictly decreasing  for $n\ge 4$.
\end{Lemma}
\begin{proof}
Suppose that $n_1>n_2\ge 4$. By Lemma \ref{poly},   $P(\theta(0,
n_1-2,0),x)-P(\theta(0, n_2-2,0),x)=
(x^{n_2}-x^{n_2-2})(x^{n_1-n_2}-1)>0$ for $x\ge \rho(\theta(0,
n_2-2,0))>1$. Thus $\rho(\theta(0, n_1-2,0))<\rho(\theta(0, n_2-2,
0))$.
\end{proof}

\begin{Lemma}\label{d2}\cite{HJ}
For  $n\ge 4$,  $\infty(2,
n-1)$ is the unique digraph in $\mathbb{B}_n$ which achieves the maximum spectral
radius,
 and  $\theta(0,n-2,0)$ is the unique $\theta$-digraph in $\mathbb{B}_n$ which
achieves the maximum spectral radius.
\end{Lemma}

\begin{Lemma}\label{l8}\cite{HJ}
If $k\ge 1$, then $\rho(\infty (k-1,l+1))>\rho(\infty(k,l))$.
\end{Lemma}


\section{Results}

To determine the
unique  digraphs with the second, the third and the fourth minimum
spectral radii respectively among strongly connected digraphs of order $n\ge 4$, we need
first to determine the unique digraphs in $\mathbb{B}_n$ with the second and the third
minimum spectral radii respectively for $n\ge 4$.

\begin{Theorem}\label{th1} For $n\ge 4$,  $\theta(1, 1,
n-4)$ and $\theta(0, 2, n-4)$ are the unique digraphs in $\mathbb{B}_n$ which achieve
the second and the third minimum spectral radii respectively.
\end{Theorem}

\begin{proof} Let $D\in \mathbb{B}_n$ with $D\neq \theta(0, 1, n-3)$. Then $D$ is a $\theta$-digraph
 or a $\infty $-digraph.
Suppose that $D$ is a $\theta$-digraph and  $D\neq \theta(1, 1, n-4)$. By Lemma \ref{add}, we have $\rho(D)\ge
\rho(\theta(0, 2, n-4))$ with equality only if $D=\theta(0, 2,
n-4)$. By Lemma \ref{poly},  $P(\theta(1, 1, n-4),x)-P(\theta(0, 2,
n-4),x)=-2x+x^2+1=(x-1)^2>0$
for $x\ge \rho(\theta(0, 2, n-4))>1$. 
 Thus $\rho(D)\ge \rho(\theta(0, 2, n-4))>\rho(\theta(1, 1, n-4))$.
If $D$ is a $\infty $-digraph,  then by Lemmas \ref{l2} and
\ref{addd},
\[
\rho(D)\ge \rho\left(\infty\left(\left\lfloor
\frac{n+1}{2}\right\rfloor,
\left\lceil\frac{n+1}{2}\right\rceil\right)\right)>\rho(\theta(0, 2,
n-4)).
\]
Now the
result follows from the first part of Lemma \ref{l2}.
\end{proof}


\begin{Theorem}\label{th3} Let $D$ be a strongly connected digraph of order $n\ge 4$ that is neither a bicyclic digraph nor  $C_n$.
Then $\rho(D)>\rho(\theta(0,2,n-4))$.
\end{Theorem}

\begin{proof} Let $C$ be a shortest directed cycle in $G$.
Obviously, $V(C)\ne V(D)$. There is a vertex $u\in V(D)\setminus
V(C)$ such that there is an arc from $u$ to some vertex, say $v$, on
$C$. Also, there is a directed path from some vertex on $C$ to $u$.
Let $w$ be a vertex on $C$ such that the distance from $w$ to $u$ in $D$ is as small as
possible. Let $P$ be such a directed path. Then $C$ and $P$ have
exactly one common vertex $w$. If $w\ne v$, then $D$ has a
proper $\theta$-subdigraph, and if $w=v$, then $D$ has  a
proper $\infty$-subdigraph.

If $D$ has a proper $\infty$-subdigraph, say $\infty(k, l)$ with
$k+l=n_1+1$ and $n_1\leq n$, then by Lemma \ref{ad}, the
second part of Lemma \ref{l2}, and Lemmas \ref{addd} and \ref{l5},
we have
\begin{eqnarray*}\rho(D)&>& \rho(\infty(k, l))\\
&\geq & \rho\left(\infty\left(\left\lfloor
\frac{n_1+1}{2}\right\rfloor,
\left\lceil\frac{n_1+1}{2}\right\rceil\right)\right)\\
&>& \rho(\theta(0,2,
n_1-4))\\
&\geq & \rho(\theta(0,2, n-4)). \end{eqnarray*}

Suppose that $D$ has a proper $\theta$-subdigraph, say
$\theta(a,b,c)$ with $a+b+c=n_2-2$ and $n_2\le n$.

\noindent {\bf Case 1.} $n_2\leq n-1$. By Lemma \ref{ad}
and the first part of Lemma \ref{l2}, we have
$$\rho(D)> \rho(\theta(a,b,c))\geq \rho(\theta(0,1, n_2-3)).$$
By Lemma \ref{poly}, $P(\theta(0, 2, n-4),x)-P(\theta(0,1,
n_2-3),x)=x^n-x^{n_2}-x^2+x=x^{n_2}(x^{n-n_2}-1)-x(x-1)\geq
x^{n_2}(x-1)-x(x-1)=(x^{n_2}-x)(x-1)>0$ for $x\ge \rho(\theta(0,1,
n_2-3))>1$. Thus $\rho(\theta(0,1, n_2-3))>\rho(\theta(0, 2, n-4))$.
Hence $\rho(D)>\rho(\theta(0, 2, n-4))$.

\noindent {\bf Case 2.} $n_2=n$ and $\theta(a,b,c)\neq \theta (0,1,
n-3)$ and
  $\theta(1,1,n-4)$. By Lemma \ref{ad}, the first part of Lemma \ref{l2}, and
Theorem \ref{th1},
$$\rho(D)>\rho(\theta(a, b, c))\geq \rho(\theta (0,2, n-4)).$$

\noindent {\bf Case 3.} $n_2=n$ and the $\theta$-subdigraph of $D$
can only be  $\theta (0,1, n-3)$ or $\theta(1,1,n-4)$. Suppose
without loss of generality that $D$ has a $\theta$-subdigraph
$\theta (0,1, n-3)$ (the proof is similar if $D$ has a
$\theta$-subdigraph $\theta(1,1,n-4)$). Let $vw$, $vu_1w$ and
$wu_1'u_2'\dots u_{n-3}'v$ be the basic directed paths of the
subdigraph  $\theta (0,1, n-3)$. We consider the possible arc(s) in
 $D$ (except the arcs in $\theta (0,1, n-3)$) as follows.

(i) $wv\not\not\in E(D)$; Otherwise,  $D$ has a $\theta$-subdigraph
$\theta(0,n-3,0)$, a contradiction.

(ii) $u_1v\not\not\in E(D)$ and $wu_1\not\not\in E(D)$; Otherwise,
$D$ has a $\theta$-subdigraph $\theta(0,n-2,0)$, a contradiction.

%

(iii) $u_1u_k'\not\in E(D)$ and  $u_{n-k-2}'u_1\not\in E(D)$ for
$2\le k\le n-3$; Otherwise, $D$ has a $\theta$-subdigraph $\theta(0,
k, n-k-2)$, a contradiction.

(iv) $vu_k'\not\in E(D)$ and $u_{n-k-2}'w \not\in E(G)$ for $1\leq
k\leq n-3$; Otherwise, $D$ has a $\theta$-subdigraph $\theta(0, k+1,
n-k-3)$, a contradiction.

(v) $u_k'v\not\in E(D)$  and $wu_{n-k-2}'\not\in E(D)$ for $1\leq
k\leq n-4$; Otherwise, $D$ has a $\theta$-subdigraph $\theta(0, 1,
k)$, a contradiction.



(vi) $u_l'u_k'\not\in E(D)$ for $1\leq k< l\leq n-3$; Otherwise, $D$
has a $\theta$-subdigraph $\theta(0, n-l+k-2, l-k-1)$, a contradiction.

(vii) $u_k'u_l'\not\in E(D)$ for $1\leq k< l-1\leq n-4$; Otherwise,
$D$ has a $\theta$-subdigraph $\theta(0, 1,n-2-(l-k))$, a contradiction.

(viii) $\{u_1u_1', u_{n-3}'u_1\}\not\subseteq E(D)$; Otherwise, $D$ has
a $\theta$-subdigraph $\theta(0,1, n-4)$, a contradiction.

From (i)--(viii), we find that besides these arcs in $\theta (0,1,
n-3)$, $D$ contains one additional arc $u_1u_1'$ or $u_{n-3}'u_1$.
Thus $D$ is isomorphic to the digraph $D'$ obtained from
$\theta(0,1,n-3)$ by adding the arc $u_1u_1'$. 
Besides the empty union and $C_n$, $\mathcal{C}(D')$ contains  two directed cycles on $n-1$ vertices. Thus
$P(D,x)=P(D',x)=x^n-2x-1$.
Obviously, $P(D,1)<0$, $P(D,2)>0$, and $P(D,x)$ is strictly
 increasing for $x\ge 1$. Thus $1<\rho(D)<2$. Similarly, $1<\rho(\theta(0,2,n-4))<2$ by Lemma \ref{poly}.
Note that $P(\theta(0,2,n-4),x)-P(D,x)=-x^2+2x>0$ for $1<x<2$. Thus
$\rho(D)>\rho(\theta(0,2,n-4))$.
\end{proof}

From Lemma \ref{l2} and Theorems \ref{th1} and \ref{th3}, we
have the following theorem.

\begin{Theorem} Among the strongly connected digraphs of order $n\ge 4$,  $\theta(0, 1, n-3)$, $\theta(1, 1, n-4)$
and $\theta(0, 2, n-4)$ are the unique digraphs that achieve the second, the third and the
fourth minimum spectral radii respectively.
\end{Theorem}

Thus we answer Problem \ref{p1} affirmatively.

Finally, we determine the unique digraphs in $\mathbb{B}_n$ with the second maximum spectral
radius for $n\ge 4$.


\begin{Theorem}\label{th2} For $n\ge 5$,   $\infty(3,
n-2)$ for $5\leq n\leq 7$, and  $\theta(0, n-2, 0)$ for $n=4$ and $n\geq 8$
are the unique digraphs in $\mathbb{B}_n$ which achieve the second maximum spectral
radius.
\end{Theorem}

\begin{proof} Obviously, $\mathbb{B}_4=\{\infty(2,3),
\theta(0,2,0), \theta(1,1,0), \theta(0,1,1)\}$. By Lemmas
\ref{l2} and \ref{d2}, and  Theorem \ref{th1}, we have
$\rho(\infty(2,3))>\rho(\theta(0,2,0))>\rho(\theta(1,1,0))>\rho(\theta(0,1,1))$.
Thus the result for $n=4$ follows.

Suppose that $n\ge 5$.
Let $D\in \mathbb{B}_n$ and $D\neq \infty(3,
n-2)$, $\theta(0, n-2, 0)$.  If $D$ is a $\theta$-digraph, then  by
the second part of  Lemma \ref{d2}, $\rho(D)<\rho(\theta(0, n-2,
0))$. If $D$ is a $\infty $-digraph and $D\neq \infty(2,n-1)$, then
by Lemma \ref{l8}, $\rho(D)<\rho(\infty(3, n-2))$. Now  by the first
part of  Lemma \ref{d2},    the second maximum spectral radius of digraphs in $\mathbb{B}_n$ is
$\max\{\rho(\theta(0, n-2, 0)), \rho(\infty(3, n-2))\}$, which is
only achieved by
 $\theta(0, n-2, 0)$
or  $\infty(3, n-2)$.

If $5\leq n\leq 7$, then by direct calculation using maple, we have
$\rho(\theta(0, n-2, 0))<\rho(\infty(3, n-2))$.

Suppose that $n\ge 8$. Let $\rho=\rho(\theta(0, n-2, 0))$. Obviously, $\rho>1$. By Lemma \ref{ll2} and direct calculation
using maple, we have $\rho\le \rho(\theta
(0,6,0))=1.1748\dots<1.175$.
For $1<x<1.175$, let $h(x)=1+x+x^2-x^3-x^4$. Since
$h'(x)=1+2x-3x^2-4x^3<0$, $h(x)$ is strictly decreasing. 
Thus $h(\rho)> h(1.175)=0.027265>0$.
By Lemma \ref{poly},
 $\rho^{n-2}=\frac{1}{\rho^2-1}$, and thus
\begin{eqnarray*}
P(\infty(3, n-2),\rho)&=&P(\infty(3, n-2),\rho)-P(\theta(0, n-2, 0),\rho)\\&=&\rho^{n-2}+1-\rho^{n-3}-\rho^2
\\&=&(\rho-1)(\rho^{n-3}-\rho-1)
\\&=&(\rho-1)\left(\frac{1}{\rho(\rho^2-1)}-\rho-1\right)
\\&=&\frac{h(\rho)}{\rho(\rho+1)}\\
&>& 0.
\end{eqnarray*}
Obviously,  $P(\infty(3, n-2),1)=-1<0$ and $P(\infty(3, n-2),x)$ is strictly increasing for $x>1$. Thus
$\rho(\infty(3, n-2))<\rho=\rho(\theta(0, n-2, 0))$.
\end{proof}

\vspace{4mm}

\noindent {\bf Acknowledgement.}
This work was supported by the National Natural
Science Foundation of China (No.~11071089) and the Research Fund for the Doctoral Program of Higher Education of China
(No.~20124407110002).

\end{document}